\documentclass[12pt]{amsart}

 \usepackage[dvips]{graphicx}
\usepackage{amsmath, amssymb}
\usepackage{amscd}
\usepackage[all]{xy}
\usepackage{pifont}

\numberwithin{equation}{section}
\newtheorem{theorem}{Theorem}[section]  
\newtheorem{theorem?}{``Theorem''}[section]  
\newtheorem{corollary}[theorem]{Corollary}
\newtheorem{proposition}[theorem]{Proposition}
\newtheorem{lemma}[theorem]{Lemma}

\theoremstyle{definition}
\newtheorem{definition}[theorem]{Definition}

\newtheorem{question}{Question}
\newtheorem{problem}[theorem]{Problem}

\theoremstyle{remark}
\newtheorem{remark}[theorem]{Remark}  

\newcommand{\R}{{\mathbb R}}
\newcommand{\C}{{\mathbb C}}

\newcommand{\N}{{\mathbb N}}
\newcommand{\Z}{{\mathbb Z}}

\renewcommand{\a}{\alpha}
\renewcommand{\b}{\beta}

\newcommand{\ord}{{\rm ord}}

\begin{document}
\title[Curves tangent to hypersurfaces]
{
On holomorphic curves tangent to real hypersurfaces 
of infinite type} 
\author{Joe Kamimoto}
\address{Faculty of Mathematics, Kyushu University, 
Motooka 744, Nishi-ku, Fukuoka, 819-0395, Japan} 
\email{
joe@math.kyushu-u.ac.jp}
\email{ }
\maketitle


\begin{abstract}
The purpose of this paper is to investigate
the geometric properties of real hypersurfaces
of D'Angelo infinite type in $\C^n$. 
In order to understand the situation 
of flatness of these hypersurfaces, 
it is natural to ask whether 
there exists a nonconstant holomorphic curve
tangent to a given hypersurface to infinite order. 
A sufficient condition for this existence is given
by using Newton polyhedra, which is an important
concept in singularity theory. 
More precisely, 
equivalence conditions are given
in the case of some model hypersurfaces. 
\end{abstract}




\section{Introduction}
Let $M$ be a ($C^\infty$ smooth) real hypersurface in $\C^n$ and  
let $p$ lie on $M$.  
Let $r$ be a local defining function for $M$ 
near $p$ ($\nabla r\neq 0$ when $r=0$).
In \cite{Dan82}, \cite{Dan93}, 
the following invariant is introduced:
\begin{equation}
\Delta_1(M, p) : = \sup_{\gamma\in\Gamma} 
\frac{\ord(r\circ \gamma)}{\ord(\gamma-p)},
\label{eqn:1.1}
\end{equation}
where $\Gamma$ denotes the set of (germs of) nonconstant 
holomorphic mappings 
$\gamma:(\C,0) \to (\C^n,p)$. 
(For a $C^{\infty}$ mapping $h:\C\to\C$ or 
$\C^n$ such that $h(0)=0$, 
let $\ord(h)$ denote the order of vanishing of $h$ at $0$.)
The invariant $\Delta_1(M, p)$ is called 
the {\it D'Angelo type} of $M$ at $p$. 
We say that $M$ is {\it of finite type} at $p$ if 
$\Delta_1(M,p)<\infty$ and {\it of infinite type}
at $p$ otherwise (the latter case will be denoted by 
$\Delta_1(M,p)=\infty$).
The class of finite type plays crucial roles in 
the study of the local regularity in
the $\bar{\partial}$-Neumann problem
over pseudoconvex domains $\Omega$ with smooth 
boundary $\partial\Omega$.
Indeed, it was shown 
by D. Catlin \cite{Cat83}, \cite{Cat87} that
$M=\partial\Omega$ is of finite type at $p$ if and only if 
a local subelliptic estimate at $p$ holds. 
From its importance, 
real hypersurfaces of finite type 
have been deeply investigated from various points of view. 

On the other hand, 
to understand the geometric properties of
real hypersurfaces of infinite type 
is also an interesting subject
in the study of several complex variables.  
These hypersurfaces contain some kind of strong flatness. 
In order to describe the geometric structure of this flatness, 
the situation of contact of holomorphic curves with the
respective hypersurface must be carefully observed.  
In this paper, we mainly consider the following question:
\begin{question}
When does there exist a nonconstant holomorphic curve
$\gamma_{\infty}$
tangent to $M$ at $p$ to infinite order?
\end{question}

Since the condition of the 
desired curve $\gamma_{\infty}$ in Question~1 
can be written as 
\begin{equation}\label{eqn:1.2}
(r\circ\gamma_{\infty})(t)=O(t^N) \quad
\mbox{
as $t\in\C \to 0$, for every $N\in\N$,}
\end{equation} 
the condition $\Delta_1(M,p)=\infty$ 
is necessary for the existence of the curve $\gamma_{\infty}$. 
It has been shown 
in \cite{Lem86}, \cite{Dan93}, \cite{FoS12} that 
when $M$ is real analytic, the above two conditions
are equivalent
(in this case, the curve $\gamma_{\infty}$ 
is contained in $M$). 
Moreover, in the case of smooth $M$,
this equivalence is also shown 
in the formal series sense 
in \cite{FLZ14}, \cite{KiN15}.  
But, in general, the sufficient direction is not true. 
Indeed, 
the nonexistense of the curve $\gamma_{\infty}$
in (\ref{eqn:1.2}) is shown in the case of 
some real hypersurfaces constructed in 
\cite{BlG77}, \cite{KiN15}, \cite{NgC19}, \cite{FoN20}. 
Understanding flatness 
on hypersurfaces of infinite type 
has been recognized to be a delicate issue. 

In this paper, 
in order to investigate the flatness of real hypersurfaces, 
we use not only holomorphic curves but also 
{\it Newton polyhedra} of defining functions, 
which plays important roles 
in singularity theory (cf. \cite{AGV85}, \cite{AGV88}). 
Approach from the viewpoint of singularity theory 
is useful in the study of types 
and there have been many works of the sort
(\cite{McN05}, \cite{Hei08}, \cite{FoS10}, 
\cite{FoS12}, \cite{Kam20}, etc.). 

First, we consider the relationship among many kinds of
infinite type hypersurfaces. 
Let 
$\Gamma^{\rm reg}:=
\{\gamma\in\Gamma:{\rm ord}(\gamma)=1\}$
and define 
$\Delta_1^{\rm reg}(M, p) : = 
\sup_{\gamma\in\Gamma^{\rm reg}} 
\{\ord(r\circ \gamma)\}$,
which is called {\it regular type} of $M$ at $p$.
Note that 
$\Delta_1^{\rm reg}(M, p) \leq \Delta_1(M, p)$.

\begin{proposition}
Let us consider the following eight conditions for 
a real hypersurface $M$ at $p$:
\begin{enumerate}
\item[(1)] $\Delta_1(M, p)=\infty$;
\item[(2)] $\Delta_1^{\rm reg}(M, p)=\infty$;
\item[(3)] 
There exists a $\gamma\in\Gamma^{\rm reg}$ 
tangent to $M$ at $p$ to infinite order;
\item[(4)]
There exists a $\gamma\in\Gamma$ 
tangent to $M$ at $p$ to infinite order;
\item[(5)] 
There exists a holomorphic coordinate 
$(z)=(z_1,\ldots,z_n)$ at $p$ 
such that $p=0$ and 
a defining function $r$ for $M$ on $(z)$ is not {\it convenient} 
(see Section~2);
\item[(6)]
There exists a holomorphic coordinate $(z)$ at $p$
such that $p=0$ and
the Newton polyhedron  
of a defining function
for $M$ on $(z)$ (see Section~2) 
takes the form 
${\mathcal N}_+(r)=
\{(\xi_1,\ldots,\xi_n)\in\R_+^n: \xi_n\geq 1 \}$;
\item[(7)]
The Bloom-Graham type of $M$ at $p$ is infinity
(i.e. there are complex submanifolds of codimension
one tangent to $M$ at $p$
to arbitrarily higher order);
\item[(8)]
$M$ is Levi-flat near $p$.
\end{enumerate}
Then, among the above eight conditions, 
the following implications hold:
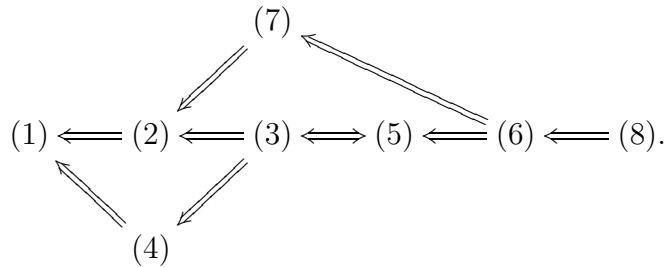
\begin{equation*}
\xymatrix{
&&(7)  \ar@{=>}[ld]  &&& \\
(1) \ar@{<=}[r]  & (2) \ar@{<=}[r] & 
(3) \ar@{<=>}[r]  & (5) \ar@{<=}[r] &
(6) \ar@{<=}[r]  \ar@{=>}[llu] & (8).  
 \\
&    (4) \ar@{=>}[lu] \ar@{<=}[ru] &  &  & & 
  }
\end{equation*}
\end{proposition}

The proof of the above proposition will be given in Section~4.1.

\begin{remark}
In the above proposition, 
for each implication with only one direction,
its opposite direction is not true
(see Remark~4.2 in Section~4 for details). 
\end{remark}

The following theorem gives a sufficient condition for 
the existence of the curve $\gamma_{\infty}$ in Question~1.
This condition is described by using Newton polyhedra
of defining functions for real hypersurfaces. 

\begin{theorem}
If there exists an ${\mathcal N}$-canonical coordinate $(z)$ for $M$ at $p$, 
then the five conditions {\rm (1)--(5)} are equivalent. 
\end{theorem}

The definition of ${\mathcal N}$-{\it canonical coordinates} will 
be given in Section 2 (Definition 2.3). 
The proof of the above theorem will be given in Section~4.2. 

The following corollary can be directly obtained 
from Theorem~1.3.

\begin{corollary}
Suppose that $\Delta_1(M,p)=\infty$. 
If there is no $\gamma\in\Gamma$ 
tangent to $M$ at $p$ to infinite order, then 
$M$ does not admit any ${\mathcal N}$-canonical coordinates at $p$.
\end{corollary}

It is easy to check that the examples of hypersurface
constructed in 
\cite{BlG77}, \cite{KiN15}, \cite{NgC19}, \cite{FoN20}
do not admit any ${\mathcal N}$-canonical coordinates.
More exactly, we will give equivalence conditions 
in more restricted cases in Sections 5 and 6.  

Next, it is seen in \cite{Kam20}
that when $M$ is the boundary of pseudoconvex 
Reinhardt domains, $M$ always admits 
an ${\mathcal N}$-canonical coordinate. 
Therefore, Theorem~1.3 implies the following. 

\begin{corollary}
Let $M$ be the boundary of pseudoconvex
Reinhardt domains (with smooth boundary) 
and let $p$ lie in $M$.
If $\Delta_1(M,p)=\infty$, then
there exists a $\gamma\in\Gamma^{\rm reg}$ 
tangent to $M$ at $p$ to infinite order. 
\end{corollary}


This paper is organized as follows. 
In Section~2, we recall the concepts: 
Newton polyhedra, ${\mathcal N}$-nondegeneracy condition 
and ${\mathcal N}$-canonical coordinates, 
which were introduced in \cite{Kam20}.
In Section~3, for the analysis later, 
we prepare appropriate coordinates on which
hypersurfaces are expressed in clear form.
Section~4 is devoted to the proof of results given in the Introduction.
More precise results are given in 
the two dimensional case in Section~5 and the higher dimensional 
case under the Bloom-Graham infinity type assumption in Section 6. 
Since the Bloom-Graham type is the same as the D'Angelo type
in the two-dimesional case, 
some results in Section~5 can be considered as special cases
of those in Section~6. 
But, they are separately explained to
make clear their difference. 
Lastly, we consider open problems in Section~7. 

{\it Notation, symbols and terminology.}\quad
\begin{itemize}
\item 
We denote 
$\Z_+:=\{n\in\Z:n\geq 0\}$ and 
$\R_+:=\{x\in\R:x \geq 0\}.$
\item
The multi-indices are used as follows.
For $z=(z_1,\ldots,z_n), \,\,
\bar{z}=(\bar{z}_1,\ldots,\bar{z}_n), 
\in\C^n$, 
$\a=(\a_1,\ldots,\a_n), \b=(\b_1,\ldots,\b_n)\in\Z_+^n$, 
define
\begin{eqnarray*}
&& 
z^{\a}:=z_1^{\a_1}\cdots z_n^{\a_n}, \,\,
\bar{z}^{\b}:=\bar{z}_1^{\b_1}\ldots\bar{z}_n^{\b_n}, \\
&&
|\a|:=\a_1+\cdots+\a_n, \quad
\a!:=\a_1 !\cdots \a_n!, \quad
0 !:=1, \\ 
&&
D^{\alpha}:=
\frac{\partial^{|\alpha|}}
{\partial z_1^{\alpha_1}\cdots \partial z_n^{\alpha_n}}, \quad
\bar{D}^{\beta}:=
\frac{{\partial}^{|\beta|}}
{\partial \bar{z}_1^{\beta_1}\cdots \partial \bar{z}_n^{\beta_n}}.
\end{eqnarray*}
\item
We always consider smooth functions, 
mappings, real hypersurfaces and complex curves 
as their respective germs without any mentioning. 
The following rings of germs of $\C$-valued functions 
are considered: 
\begin{itemize}
\item 
$C^{\infty}_0(\C^n)$ 
is the ring of germs of $C^{\infty}$ functions 
at the origin in $\C^n$. 
\item 
${\mathcal O}_0(\C)$ is the ring of germs of 
holomorphic functions at the origin in $\C$.
\end{itemize}
\item
We always take a {\it good} parametrization 
for corves $\gamma\in\Gamma$ 
without any mentioning. 
That is to say, a point on the curve, defined 
by $t\mapsto \gamma(t)$, 
corresponds to only one value of $t$. 
For example, $(t,t^2)$ is good, but 
$(t^2,t^4)$ is not good. 
\item
We use the words {\it pure terms} for any harmonic
polynomial and {\it mixed terms}
for any sum of monomials that are
neither holomorphic nor anti-holomorphic. 
\end{itemize}


\section{Newton polyhedra for real hypersurfaces}

Let us define the Newton polyhedron of 
a real-valued smooth function $F$ defined near the origin 
in $\C^n$.  
The Taylor series expansion of $F$ at the origin is 
\begin{equation}\label{eqn:2.1}
F(z,\bar{z}) \sim \sum_{\a,\, \b \in \Z_+^n} 
C_{\a \b} z^\a \bar{z}^{\b} \quad \mbox{ with 
$C_{\a \b} = \dfrac{1}{\a ! \b !} 
D^{\alpha}\bar{D}^{\beta}F(0,0)$.}
\end{equation}
The {\it Newton polyhedron} of $F$ is defined by 
\begin{equation*}
{\mathcal N}_+(F) = 
\mbox{The convex hull of }
\left(\bigcup_{\a+\b\in S(F)}(\a+\b + \R_{+}^n)\right),
\end{equation*}
where $S(F)=\{\a+\b\in\Z_+^n:C_{\a \b}\neq 0\}.$
The {\it Newton diagram} ${\mathcal N}(F)$ of $F$ is defined to be 
the union of the compact faces of ${\mathcal N}_+(F)$.
We use coordinates $(\xi)=(\xi_1,\ldots,\xi_n)$ for points in the plane
containing the Newton polyhedron.  
The following classes of functions $F$ 
simply characterized by using their Newton polyhedra 
often appear in this paper:
\begin{itemize}
\item
$F$ is called to be {\it flat} at $0$
if ${\mathcal N}_+(F)$ is an empty set.
\item
$F$ is called to be {\it convenient} at $0$
if ${\mathcal N}_+(F)$ meets every coordinate axis.
\end{itemize}

Let $(z)=(z_1,\ldots,z_n)$ be a holomorphic coordinate
around $p$ such that $p=0$. 
Let $r$ be a local defining function for $M$ near $p$
on the coordinate $(z)$. 
For a given tuple $(M,p;(z))$, we define
a quantity $\rho_1(M,p;(z))\in\Z_+ \cup\{\infty\}$ as follows. 
If $r$ is convenient, then let 
\begin{equation*}\label{eqn:}
\rho_1(M,p;(z)):=\max\{\rho_j(r):j=1,\ldots,n\},
\end{equation*}
where  $\rho_j(r)$ is the coordinate of 
the point at which the Newton diagram ${\mathcal N}(r)$ 
intersects the $\xi_j$-axis.
Otherwise, let $\rho_1(M,p;(z)):=\infty$.
We remark that $\rho_1(M,p;(z))$ depends on 
the chosen coordinate $(z)$, but it is independent of 
the choice of defining functions after fixing a coordinate.
Considering the curves 
$\gamma_j(t)=(0, \ldots, \stackrel{(j)}{t}, \ldots, 0)
\in\Gamma^{\rm reg}$ for $j=1,\ldots,n$,
we can see that the inequality
$
\rho_1(M,p;(z))\leq \Delta_1^{\rm reg}(M,p)
$
always holds.

Next, let us introduce an important concept 
``${\mathcal N}$-nondegeneracy condition''
on a smooth function $F$ defined near the origin 
in $\C^n$. 

Let $\kappa$ be a compact face of ${\mathcal N}_+(F)$.
The $\kappa$-{\it part of} $F$ is the polynomial defined by 
\begin{equation}\label{eqn:2.2}
F_{\kappa}(z,\bar{z})=
\sum_{\a+\b \in \kappa} 
C_{\a \b} z^\a \bar{z}^{\b}.
\end{equation}
The set of holomorphic curves ${\Gamma}_{\kappa}$ 
is defined by
\begin{equation*}\label{eqn:1.7}
\begin{split}
{\Gamma}_{\kappa}:=
\{(c_1 t^{a_1},\ldots,c_n t^{a_n}): 
c\in(\C\setminus\{0\})^n, t\in\C, 
\mbox{ $a\in\N^n$ determines $\kappa$ }\},
\end{split}
\end{equation*}
where
$c=(c_1,\ldots,c_n)\in(\C\setminus\{0\})^n$, 
$a=(a_1,\ldots,a_n)\in\N^n$ and 
``$a\in\N^n$ determines $\kappa$'' means that 
the set
$\{\xi\in{\mathcal N}_+(F):
\sum_{j=1}^n a_j \xi_j =l\}$ 
coincides with the face $\kappa$ 
for some $l\in\N$.
\begin{definition}
Let $\kappa$ be a compact face of ${\mathcal N}_+(F)$.
The $\kappa$-part $F_{\kappa}$ of $F$ 
is said to be {\it ${\mathcal N}$-nondegenerate} if 
\begin{equation*}\label{eqn:1.8}
F_{\kappa}\circ \gamma \not\equiv 0 
\quad 
\mbox{ for any $\gamma\in {\Gamma}_{\kappa}$}.
\end{equation*}
A function $F$ is said to be {\it ${\mathcal N}$-nondegenerate}, 
if $F_{\kappa}$ is ${\mathcal N}$-nondegenerate 
for every compact face 
$\kappa$ of ${\mathcal N}_+(F)$. 
\end{definition}
The above concept is analogous to the nondegeneracy 
condition introduced by Kouchnirenko \cite{Kou76}, 
which plays important roles in the singularity theory. 
Detailed properties of this condition 
are explained in \cite{Kam20}.

\begin{definition}
A holomorphic coordinate $(z)$ at $p$ is said to be {\it canonical}
for $M$ at $p$ if a local defining function $r$ for $M$
on $(z)$ is ${\mathcal N}$-nondegenerate.
\end{definition}

The following relationship 
\begin{equation}\label{eqn:2.3}
\Delta_1(M, p) \geq \Delta_1^{{\rm reg}}(M, p)
\geq \rho_1(M,p;(z))
\end{equation}
is always established for every coordinate $(z)$ at $p$. 
The following theorem shows that the equalities 
in (\ref{eqn:2.3}) are satisfied under 
the ${\mathcal N}$-nondegeneracy condition.
\begin{theorem}[\cite{Kam20}]
\label{thm:2.3}
If there exists an ${\mathcal N}$-canonical coordinate $(z)$ at $p$, 
then the following equalities hold: 
\begin{equation}\label{eqn:2.4}
\Delta_1(M, p) = \Delta_1^{{\rm reg}}(M, p)=\rho_1(M,p;(z)).
\end{equation}
\end{theorem}
Note that the above theorem is valid for the infinite type case.
From the above theorem, 
the existence of ${\mathcal N}$-canonical coordinates 
implies that both values of  
$\Delta_1(M, p)$ and $\Delta_1^{{\rm reg}}(M, p)$
can be directly seen from geometrical
Newton data of $M$ at $p$. 


\section{Standard coordinates}

Let $M$ be a real hypersurface in $\C^{n+1}$ 
($n\geq 1$) and 
let $p$ lie in $M$. 

It follows from Taylor's formula that 
there exists a holomorphic coordinate 
$(z,w):=(z_1,\ldots,z_n,w)$ at $p$
on which a local defining function $r$ for $M$ 
is expressed near the origin
as in the following form: 
\begin{equation}\label{eqn:3.1}
r(z,w,\bar{z},\bar{w})  
=2{\rm Re}(w)+ 
F(z,\bar{z})+R_1(z,\bar{z})\cdot{\rm Im}(w)+
R_2(z,w,\bar{z},\bar{w}),
\end{equation}
where
\begin{enumerate}
\item 
$F\in C^{\infty}_0(\C^n)$ satisfies that
$F(0,0)=0$ and $|\nabla F(0,0)|=0$;
\item 
$R_1\in C^{\infty}_0(\C^n)$ 
and $R_2\in C^{\infty}_0(\C^{n+1})$.  
Moreover, $R_1$, $R_2$ satisfies that
${\rm ord}(R_1)\geq 1$ and 
$|R_2|\leq C |{\rm Im}(w)|^2$
near $(z,w)=0$ where $C$ is a positive constant
independent of $(z,w)$.
\end{enumerate}
Of course, there may be many such coordinates, 
which are said to be {\it standard} 
for $M$ at $p$.

Furthermore, if there exists 
a holomorphic coordinate 
$(z,w):=(z_1,\ldots,z_n,w)$ around $p$
on which a local defining function $r$ for $M$ 
is expressed near the origin
as in the {\it model} form: 
\begin{equation}\label{eqn:3.2}
r(z,w,\bar{z},\bar{w})  
=2{\rm Re}(w)+ 
F(z,\bar{z}),
\end{equation}
where $F$ is as in (\ref{eqn:3.1}), 
then $(z,w)$ is called a 
{\it good (standard) coordinate} 
for $M$ at $p$. 
Note that 
good coordinates do not always exist
for all hypersurfaces.


\section{Proofs of results in the Introduction}

Let $M$ be a real hypersurface in $\C^n$ and 
let $p$ lie in $M$. 

From Theorem~2.3, 
under the ${\mathcal N}$-nondegeneracy condition,
a defining function for $M$ is convenient
if and only if $M$ is of finite type at $p$. 
Roughly speaking, 
when $M$ does not satisfy the convenience condition, 
$M$ contains a flat part in one direction.
The following lemma exactly explains such situation.

\begin{lemma}
Let $r$ be a local defining function for $M$ near $p$
on some holomorphic coordinate $(z)$.
Then the following four conditions are equivalent:
\begin{enumerate}
\item 
$r$ is not convenient (i.e. $\rho_1(M,p;(z))=\infty$);
\item 
There exists $k\in\{1,\ldots,n\}$ such that
${\mathcal N}_+(r)$ does not intersect $\xi_k$-axis;
\item 
There exists $k\in\{1,\ldots,n\}$ such that
the Taylor series of $r$ does not contain 
any term consisting of $z_k$ and $\bar{z}_k$ only;
\item 
There exists 
$k\in\{1,\ldots,n\}$ such that
$(r\circ\gamma_k)(t)=O(t^N)$ 
for every $N\in\N$, where
$\gamma_k=
(0,\ldots,\stackrel{(k)}{t},\ldots,0)\in
\Gamma^{\rm reg}$.
\end{enumerate}
\end{lemma}

The following lemma expresses a property of 
flat hypersurfaces by using the language of Newton polyhedra. 

\begin{lemma}
The following two conditions are equivalent:
\begin{enumerate}
\item 
There exists a standard coordinate for $M$ at $p$
such that $F$ in (\ref{eqn:3.1}) is flat 
at the origin;
\item
There exists a holomorphic coordinate $(z)$ at $p$
such that $p=0$ and
the Newton polyhedron 
of a defining function $r$
for $M$ on $(z)$ takes the form: 
${\mathcal N}_+(r)=
\{(\xi_1,\ldots,\xi_n)\in\R_+^n: \xi_n\geq 1 \}$.
\end{enumerate}
\end{lemma}

Since the above two lemmas are easy, 
their proofs are omitted.  

\subsection{Proof of Proposition~1.1}

(3) $\Longleftrightarrow$ (5)
was stated in Lemma~4.1.
(6) $\Longrightarrow$ (7)
is easy to see from Lemma~4.2.
(7) $\Longrightarrow$ (2)
is shown in \cite{FoN20}. 
(Lemma~5 in \cite{FoN20} states
(7) $\Longrightarrow$ (1), but 
its proof actually implies the above stronger implication.) 
The other implications are obvious.

\begin{remark}
Let us consider the converse of each implication 
in Proposition~1.1.
We give an example of the hypersurface 
showing that the converse does not hold 
in the implication with only one direction. 
All our examples of hypersurface $M$ in $\C^{n+1}$ 
($n\in\N$) are pseudoconvex and admit a good coordinate, 
i.e., 
$M$ is locally expressed as in the model form:
$
2{\rm Re}(w)+F(z,\bar{z})=0, 
$ as in (\ref{eqn:3.2}).
Thus, 
we only write a function $F(z,\bar{z})$ 
for each case. 
In the two-dimensional case,  
the converses of some implications are also true 
(see Section~5). 
In these cases,
counterexamples must be constructed 
in $\C^n$ with $n\geq 3$.  

\begin{itemize}
\item ((1) $\not\Rightarrow$ (2)) \quad
$F(z,\bar{z})=
|z_1^3-z_2^2|^2+|z_3|^2+\cdots+|z_n|^2$ \quad
($n\geq 2$);
\item ((2) $\not\Rightarrow$ (3)) \quad
$F(z,\bar{z})=
f(z_1,\bar{z}_1)+|z_2|^2+\cdots+|z_n|^2$ \quad
($n\in\N$);
\item ((1) $\not\Rightarrow$ (4))  \quad
$F(z,\bar{z})=
f(z_1,\bar{z}_1)+|z_2|^2+\cdots+|z_n|^2$
 \quad
($n\in\N$);
\item ((4) $\not\Rightarrow$ (3)) \quad
$F(z,\bar{z})=
|z_1^3-z_2^2|^2+|z_3|^2+\cdots+|z_n|^2$ \quad
($n\geq 2$);
\item ((5) $\not\Rightarrow$ (6))  \quad
$F(z,\bar{z})=|z_n|^2$ \quad
($n\geq 2$);
\item ((2) $\not\Rightarrow$ (7)) \quad
$F(z,\bar{z})=|z_n|^2$ \quad
($n\geq 2$);
\item ((7) $\not\Rightarrow$ (6))  \quad
$F(z,\bar{z})=f(z_n,\bar{z}_n)$ \quad
($n\in\N$);
\item ((6) $\not\Rightarrow$ (8)) \quad
$F(z,\bar{z})=
\exp(-|z_n|^{-2})$ \quad
($n\in\N$).
\end{itemize}
Here $f\in C_0^{\infty}(\C)$ is a subharmonic 
function admitting a divergent Taylor series at the origin. 
It was constructed in \cite{FoN20} 
(see Remark~5.8).

It is shown in \cite{KiN15}, \cite{NgC19}
(see also Corollary~6.4 in this paper) 
that the second example shows 
(2) $\not\Rightarrow$ (3) and 
(1) $\not\Rightarrow$ (4) in the two-dimensional case. 
The higher dimensional case can be easily shown. 
It is easy to check that the other hypersurfaces
are counterexamples. 

\end{remark}

\subsection{Proof of Theorem~1.3}

It suffices to show the implication: (1) $\Longrightarrow$ (5).

Let $(z)$ be an ${\mathcal N}$-canonical coordinate for $M$ at $p$ 
and let $r$ be a local defining function for $M$ near $p$
on the coordinate $(z)$.
From Theorem~2.3, 
the condition (1) implies $\rho_1(M,p;(z))=\infty$, 
which is equivalent to the condition (5) from Lemma~4.1. 

%

\section{Two dimensional case}

In this section, 
we more precisely consider Question~1 in the Introduction 
in the case when a real hypersurface $M$ is in $\C^2$.
Let $p\in M$. 

In the two dimensional case, 
many implications in Proposition~1.1 can be refined
by equivalences. 

\begin{proposition}
Let $M$ be a real hypersurface in $\C^2$ and 
let $p$ lie in $M$.
Among the above eight conditions for $M$ at $p$ 
in Proposition~1.1, 
the following implications hold:
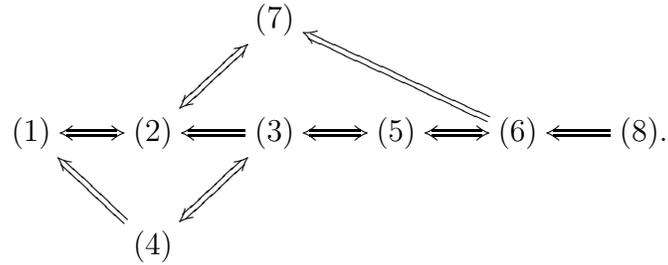
\begin{equation*}
\xymatrix{
&&(7)  \ar@{<=>}[ld]  &&& \\
(1) \ar@{<=>}[r]  & (2) \ar@{<=}[r] & 
(3) \ar@{<=>}[r]  & (5) \ar@{<=>}[r] &
(6) \ar@{<=}[r]  \ar@{=>}[llu] & (8).  
 \\
&    (4) \ar@{=>}[lu] \ar@{<=>}[ru] &  &  & & 
  }
\end{equation*}
\end{proposition}

\begin{proof}
(5) $\Longrightarrow$ (6) is obvious.
It is known in \cite{Dan93} that
(1), (2), (7) are equivalent. 
(4) $\Longrightarrow$ (3) 
will be shown in Lemma~5.2, below.
\end{proof}

\begin{lemma}
If $\gamma\in\Gamma$ is tangent to 
$M\subset \C^2$ at $p$
to infinite order, 
then $\gamma\in\Gamma^{\rm reg}$. 
\end{lemma}

\begin{proof}
Let $(z,w)$ be a standard coordinate for $M$ 
at $p$ such that $M$ is expressed as in (\ref{eqn:3.1}).

First, if $F$ is flat at $0$, then the curve $\gamma$,  
satisfying 
$(r\circ\gamma)(t)=O(t^N)$ for every $N\in\N$, 
essentially takes the form: 
$\gamma(t)=(t,0)$.
This curve is regular.

Next, let us consider the case when ${\rm ord}(F)=m$ 
with some $m\in\N$.
In this case, the Newton polyhedron of $r$ takes 
the form
${\mathcal N}_+(r)=
\{\xi\in\R_+^2:\xi_1/m+\xi_2\geq 1\}$.
Let 
$\kappa:=\{\xi\in\R_+^2:\xi_1/m+\xi_2= 1\}$,
which is the only compact facet of ${\mathcal N}_+(r)$.
If $r_{\kappa}$ were ${\mathcal N}$-nondegenerate, 
then Theorem~2.3 implies $\Delta_1(M,p)=m<\infty$,
which is a contradiction. 
If $r_{\kappa}$ is not ${\mathcal N}$-nondegenerate,
then a desired curve must be written in the form 
$\gamma(t)=(t,ct^m+O(t^{m+1}))$ with $c\neq 0$
from the definition of the ${\mathcal N}$-nondegeneracy. 
This curve is also regular. 
\end{proof}


From Theorem~2.3,
we can see the following. 

\begin{theorem}
Let $M\subset\C^2$.
If there exists an ${\mathcal N}$-canonical coordinate 
for $M$ at $p$, then
the seven conditions (1)--(7) in Proposition~5.1 
are equivalent. 
\end{theorem}

The following lemma is essentially the same 
as Theorem 2 in \cite{KiN15}
(see also \cite{Kol05}).  

\begin{lemma}
Let $M\subset\C^2$. 
If $\Delta_1(M,p)=\infty$, then
the Taylor series of $F$ in (\ref{eqn:3.1}) 
at the origin consists of only pure terms
on every standard coordinate. 
\end{lemma}

\begin{proof}
Let us assume the existence of 
a standard coordinate $(z,w)$ 
on which 
the Taylor series of $F$  
contains a mixed term. 
Let $m$ be the minimum order of the 
mixed terms of Taylor series of $F$ and 
let $p_m(z,\bar{z})$ denote the sum of 
the mixed terms of order $m$. 

It is easy to construct a standard coordinate $(z,w^*)$
on which $M$ can be expressed by the equation: 
\begin{equation*}
2{\rm Re}(w^*)+p_m(z,\bar{z})+Q(z,\bar{z})+
R_1^*\cdot {\rm Im}(w^*)
+R_2^*=0,
\end{equation*}
where $Q\in C_0^{\infty}(\C)$ with 
${\rm ord}(Q)\geq m+1$ and 
$R_1^*,R_2^*$ have the same properties as those of 
$R_1,R_2$ in (\ref{eqn:3.1}), respectively. 
Since $2{\rm Re}(w^*)+p_m(z)$ is 
${\mathcal N}$-nondegenerate, 
the above $(z,w^*)$ is canonical for $M$ at $p$. 
It follows from Theorem 2.3 that 
$\Delta_1(M,p)=m<\infty$, which is a contradiction.
\end{proof}

From the above lemma, 
when $\Delta_1(M,p)=\infty$,  
the Taylor series of $F$ can be expressed as 
$2\sum_{j=2}^{\infty} {\rm Re}(c_j z^j)$
where $c_j\in\C$. 
When the hypersurface $M$ and $p\in M$ 
are fixed, 
the sequence of complex numbers 
$\{c_j\}_{j\in\N}$ is determined by 
the chosen standard coordinates $(z,w)$ only. 
For a given coordinate $(z,w)$, 
we define the formal power series:
\begin{equation}\label{eqn:5.1}
{\mathcal S}(z)=
\sum_{j=2}^{\infty} c_j z^j
\quad (z\in\C).
\end{equation} 
If $M$ is real analytic, then 
the series ${\mathcal S}(z)$
converges near the origin.  
But,
its converse is not always true.  
From Cauchy-Hadamard's formula, 
positivity of the convergence radius of the power series
${\mathcal S}(z)$ is
equivalent to the condition
$\limsup_{j\to\infty} |c_j|^{1/j}<\infty$.

Hereafter in this section, 
we only consider the case when 
the hypersurface admits a good coordinate at $p$. 
That is to say, $M$ can be expressed as in the model form
\begin{equation}\label{eqn:5.2}
r(z,w,\bar{z},\bar{w})=2{\rm Re}(w)+F(z,\bar{z})=0,
\end{equation}
where $F$ is the same as that in (\ref{eqn:3.1}). 

The following theorem gives equivalence conditions for 
Question~1 in the Introduction 
in the two-dimensional model case. 

\begin{theorem}
Let $M$ be a real hypersurface in $\C^2$  
admitting a good coordinate at $p$ as in (\ref{eqn:5.2}).
If $\Delta_1(M,p)=\infty$, then
the following three conditions are equivalent.
\begin{enumerate}
\item 
There exists a $\gamma\in\Gamma^{\rm reg}$
tangent to $M$ at $p$ to infinite order; 
\item 
There exists a good coordinate $(z,w)$ 
for $M$ at $p$ on which
$F$ is flat;
\item
There exists a good coordinate $(z,w)$ 
for $M$ at $p$ such that the convergence radius of 
the power series 
${\mathcal S}(z)$ in (\ref{eqn:5.1}) is 
positive.
\end{enumerate}
\end{theorem}

\begin{proof}
(ii) $\Longrightarrow$ (i) and 
(ii) $\Longrightarrow$ (iii) are obvious.  

First, let us show (iii) $\Longrightarrow$ (ii).
From (iii), 
the power series 
${\mathcal S}(z)$ 
can be 
regarded as a holomorphic function defined
on some open neighborhood of $z=0$.
Putting $w^*=w-{\mathcal S}(z)$, 
we can express the hypersurface $M$ on the good coordinate
$(z,w^*)$ by the equation
$
2{\rm Re}(w^*)+F^*(z,\bar{z})=0$, 
where $
F^*(z,\bar{z})=
F(z,\bar{z})-2{\rm Re}({\mathcal S}(z))$. 
Since $F(z)-2{\rm Re}({\mathcal S}(z))$ is flat at $z=0$, 
the existence of coordinate in (ii) is shown. 

Next, let us show (i) $\Longrightarrow$ (iii).
We may assume that the regular holomorphic curve 
in (i) can be expressed as 
$\gamma(t)=(t,-h(t))$ 
on a good coordinates $(z,w)$ satisfying that 
$(r\circ \gamma)(t)=O(t^N)$ for every $N\in\N$. 
Here $h$ is holomorphic  near the origin 
and satisfies $h(0)=0$.
Let $\sum_{j=1}^{\infty} a_j t^j$ be 
the Taylor series of $h$ at $t=0$, 
which converges on an open neighborhood of $t=0$.
Now, after the first finite sum of the Taylor series 
of $h$ and $F$
are substituted into the equations
$(r\circ \gamma)(t)=
-2{\rm Re}(h(t))+F(t,\bar{t})=0$, we have
\begin{equation*}
{\rm Re}\left(
\sum_{j=1}^{N}a_j t^j 
\right)-
\sum_{j=2}^{N} {\rm Re}(c_j t^j)=O(t^{N+1}),
\end{equation*}
for every $N\in\N$. 
From the above equality, we can see that 
$c_j=a_j$ for every $j\in\N$. 
This means that ${\mathcal S}(t)$ converges 
and ${\mathcal S}(t)=h(t)$
on an open neighborhood of $t=0$.
\end{proof}

\begin{remark}
The condition (iii) in Theorem~5.5 is weaker
than the condition: $M$ is real analytic near $p$.
Indeed, as mentioned in the Introduction, 
the real analyticity of $M$ implies the existence
of desired curves in (i) in the theorem 
(\cite{Lem86}, \cite{Dan93}, \cite{FoS12}). 
The hypersurface $M$ defined by 
$2{\rm Re}(w)+e^{-1/|z|^2}=0$ 
with $p=(0,0)$ satisfies the conditions (ii) and (iii),
but $M$ is not real analytic at $p$. 
\end{remark}

In \cite{BlG77}, \cite{KiN15}, \cite{NgC19}, 
smooth real hypersurfaces in $\C^2$ 
admitting no curve $\gamma_{\infty}$ in 
(\ref{eqn:1.2}) are constructed.
It follows from Theorem~5.5 that
there are many such hypersurfaces.

\begin{corollary}
Let $M$ be a real hypersurface defined by 
the equation (\ref{eqn:5.2}). 
If $\Delta_1(M,p)=\infty$,  
then the following two conditions are equivalent. 
\begin{enumerate}
\item 
There is no $\gamma\in\Gamma$ 
tangent to $M$ at $0$ to infinite order; 
\item 
The convergence radius of the power series
${\mathcal S}(z)$ equals zero.
\end{enumerate}
\end{corollary}

\begin{remark}
From (ii) in the above corollary, 
it is quite easy to construct smooth hypersurfaces 
satisfying the condition (i). 
Let $\{c_j\}_{j\in\N}$ be a sequence of 
complex numbers such that
the convergence radius of 
$\sum_{j=2}^{\infty}c_j z^j$ 
is zero.
By using a classical lemma of E. Borel 
(cf. \cite{Nar85}, Theorem~1.5.4,  
or  \cite{Hor90}, Theorem~1.2.6), 
for the formal power series of $(x,y)\in\R^2$ 
($z=x+iy$):
\begin{equation}\label{eqn:5.3}
\sum_{j=2}^{\infty} {\rm Re}(c_j z^j)
=\sum_{(j,k)\in\Z_+^2}
C_{jk}x^j y^k \mbox{\,\, with $C_{jk}\in\R$}, 
\end{equation}
there exists a real-valued $C^{\infty}$ function $f$
defined near the origin in $\R^2$ 
whose Taylor series at the origin is 
(\ref{eqn:5.3}).
Then the real hypersurface defined by 
$2{\rm Re}(w)+f(z,\bar{z})=0$ 
satisfies the condition (i). 
We remark that the above hypersurface cannot always be 
uniquely determined from the sequence
$\{c_j\}_{j\in\N}$. 

In \cite{KiN15}, \cite{NgC19}, 
similar examples of hypersurfaces have been found. 
Since their constructions 
do not use a lemma of E. Borel directly, 
they look much more elaborate. 

Furthermore, 
it is shown in \cite{FoN20} that
when $\{c_j\}_{j\in\N}$ is an increasing sequence of positive real
numbers, the above function $f$ can be selected to be
a subharmonic function on $\C$. 
Therefore, there are many pseudoconvex real hypersurfaces
satisfying the condition (i) in the corollary. 
\end{remark}


\section{Higher dimensional case}

In this section, we generalize results 
given in the previous section to the higher dimensional case.
Let $M$ be a real hypersurface in $\C^{n+1}$ 
($n\geq 1$)
and let $p$ lie in $M$.
In this section, we always consider the case when
$M$ is of Bloom-Graham infinite type at $p$.
First, let us recall the exact definition 
of the Bloom-Graham type in \cite{BlG77}.
We remark that the following definition
is an equivalence condition for their original type.
This equivalence is also shown in \cite{BlG77}.

\begin{definition}
Let ${\mathcal X}$ be a set of 
$n$-dimensional complex submanifolds 
containing $p$. 
We say that 
the Bloom-Graham type of $M$ is $m$ $(<\infty)$,
if there is a $X\in{\mathcal X}$ tangent to $M$ 
at $p$ to order $m$ but no $X\in{\mathcal X}$ 
tangent to a higher order.
Otherwise, we say that 
the Bloom-Graham type of $M$ is infinity at $p$
(see the condition (7) in Proposition~1.1).
\end{definition}

In the case of Bloom-Graham infinity type,
a similar property to Lemma~5.4 can be seen.
\begin{lemma}
If the Bloom-Graham type of $M$ at $p$ is infinity, then
the Taylor series of $F$ in (\ref{eqn:3.1}) 
at the origin consists of only pure terms
on every standard coordinate (see Section~3). 
\end{lemma}

\begin{proof}
Let us assume the existence of 
a standard coordinate $(z,w)$ 
on which 
the Taylor series of $F$  
contains a mixed term. 
Let $m$ be the minimum order of the 
mixed terms of Taylor series of $F$.

It is easy to construct 
a new standard coordinate $(z,w^*)$
on which $M$ can be expressed by the equation
\begin{equation*}
2{\rm Re}(w^*)+P_m(z,\bar{z})+Q(z,\bar{z})+
R_1^*\cdot {\rm Im}(w^*)
+R_2^*=0,
\end{equation*}
where 
$P_m$ is a non-zero mixed homogenous polynomial of degree $m$ 
without pure terms, 
$Q\in C_0^{\infty}(\C^n)$ with 
${\rm ord}(Q)\geq m+1$ and 
$R_1^*,R_2^*$ have the same properties as those of 
$R_1,R_2$ in (\ref{eqn:3.1}), respectively. 
It follows from the definition that the Bloom-Graham type
of $M$ at $p$ equals $m<\infty$, which is a contradiction.
\end{proof}

From Lemma~6.2, 
the Taylor series of $F$ at the origin in $\C^n$ 
can be expressed as 
$2\sum_{\alpha\in\Z_+^n} 
{\rm Re}(c_{\alpha} z^{\alpha})$
where $c_{\alpha}\in\C$. 
When the hypersurface $M$ and $p\in M$ 
are fixed, 
the sequence $\{c_{\alpha}\}_{\alpha\in\Z_+^n}$ 
is determined by 
the chosen standard coordinates $(z,w)$. 
For a given coordinate $(z,w)$, 
we define the formal power series
\begin{equation}\label{eqn:6.1}
{\mathcal S}(z)=
\sum_{\alpha\in\Z_+^n} c_{\alpha} z^{\alpha}
\quad (z\in\C^n).
\end{equation} 

Hereafter in this section, 
we only consider the case when 
the hypersurface admits a good coordinate at $p$. 
That is to say, $M$ can be expressed as in the model form
\begin{equation}\label{eqn:6.2}
r(z,w,\bar{z},\bar{w})=2{\rm Re}(w)+F(z,\bar{z})=0,
\end{equation}
where $F$ is the same as in (\ref{eqn:3.1}). 

Let $\hat{\Gamma}$ be the set 
of (germs of) nonconstant holomorphic curves
$\hat{\gamma}=
(\hat{\gamma}_1,\ldots,\hat{\gamma}_n)
:(\C,0) \to (\C^n,0)$. 
Let $\sum_{k=1}^{\infty} a_{jk} t^k$ be 
the Taylor series of $\hat{\gamma}_j$ 
for $j=1,\ldots,n$.
After these Taylor series are substituted into 
${\mathcal S}(z_1,\ldots,z_n)$ with 
$z_j=\hat{\gamma}_j(t)$, 
a formal computation gives the formal power series, 
denoted by $({\mathcal S}\circ\hat{\gamma})(t)$, 
in the following.
\begin{equation}\label{eqn:6.3}
({\mathcal S}\circ\hat{\gamma})(t)
={\mathcal S}
(\hat{\gamma}_1(t),\ldots,\hat{\gamma}_n(t))
=\sum_{\alpha\in\Z_+^n} c_{\alpha}
\prod_{j=1}^n 
\left(
\sum_{k=1}^{\infty}a_{jk} t^k
\right)^{\alpha_j}
=:\sum_{j=2}^{\infty}c_j t^j.
\end{equation}
The relationship between $F$ and ${\mathcal S}$ 
implies that
\begin{equation}\label{eqn:6.4}
(F\circ\hat{\gamma})(t)=
2\sum_{j=2}^{N} {\rm Re}(c_j t^j)+O(t^{N+1}),
\end{equation}
for every $N\in\N$.

The following theorem is a natural generalization 
of Theorem~5.6. 
Let $\Gamma$ denote the set of nonconstant 
holomorphic mappings 
$\gamma:(\C,0) \to (\C^{n+1},p)$. 

\begin{theorem}
Let $M$ be a real hypersurface in $\C^{n+1}$ ($n\geq 1$)
admitting a good coordinate at $p$ as in (\ref{eqn:6.2}). 
If the Bloom-Graham type of $M$ at $p$ is infinity, then
the following three conditions are equivalent.
\begin{enumerate}
\item 
There exists a $\gamma\in\Gamma$
tangent to $M$ at $p$ to infinite order; 
\item 
There exists a good coordinate $(z,w)$ 
for $M$ at $p$ on which
$(F\circ\hat{\gamma})(t)$ is flat at $t=0$ for some 
$\hat{\gamma}\in\hat{\Gamma}$; 
\item
There exists a good coordinate $(z,w)$ 
for $M$ at $p$ such that the formal power series 
$({\mathcal S}\circ\hat{\gamma})(t)$ 
in (\ref{eqn:6.3})
converges on an open neighborhood of $t=0$
for some $\hat{\gamma}\in\hat{\Gamma}$, 
where 
${\mathcal S}$ is as in (\ref{eqn:6.1}).
\end{enumerate}
\end{theorem}

\begin{proof}
First, let us show (ii) $\Longrightarrow$ (i).
Let $\gamma(t)=(\hat{\gamma}(t),0)\in\Gamma$,  
where $\hat{\gamma}$ is as in (ii),
and let $r$ be a defining function for $M$ 
as in (\ref{eqn:6.2}).
Then we have
$
(r\circ \gamma)(t)=(F\circ\hat{\gamma})(t) 
=O(t^N)$
for every $N\in\N$, which implies (i). 

Second, let us show (iii) $\Longrightarrow$ (ii).
Let $I:=\{j:\hat{\gamma}_j\not\equiv 0\}$.
Define the map $T_{I}:\C^n\to\C^n$ 
by $(w_1,\ldots,w_n)=T_I(z_1,\ldots,z_n)$
where $w_j=z_j$ if $j\in I$ and $w_j=0$ otherwise.
It follows from Abel's lemma (cf. \cite{FrG02}) that 
the convergence of 
$({\mathcal S}\circ\hat{\gamma})(t)$ for some $t\neq 0$
implies that the power series $({\mathcal S}\circ T_I)(z)$
converges on an open neighborhood of $z=0$, 
which means that ${\mathcal S}\circ T_I$
can be regarded as a holomorphic function there. 
Letting $w^*=w-({\mathcal S}\circ T_I)(z)$, 
we can express the hypersurface $M$ on the good coordinate
$(z,w^*)$ by the equation:
$
2{\rm Re}(w^*)+F^*(z,\bar{z})=0$,
where 
$F^*(z,\bar{z})=F(z,\bar{z})-
2{\rm Re}(({\mathcal S}\circ T_I)(z))$. 
By using the equality
${\mathcal S}\circ T_I \circ \hat{\gamma}=
{\mathcal S}\circ \hat{\gamma}$, 
$(F^*\circ\hat{\gamma})(t)=(F\circ\hat{\gamma})(t)-
2{\rm Re}({\mathcal S}\circ\hat{\gamma})(t)$ 
is flat at $t=0$, 
which implies (ii). 

Third, let us show (i) $\Longrightarrow$ (iii).
We may assume that a holomorphic curve 
in (i) can be expressed as 
$\gamma(t)=(\hat{\gamma}(t),-h(t))$, 
where $\hat{\gamma}\in\hat{\Gamma}$ and 
$h\in{\mathcal O}_0(\C)$ 
with $h(0)=0$,
on a good coordinates $(z,w)$.
Note that (i) is equivalent to the condition
$(r\circ \gamma)(t)=O(t^N)$ for every $N\in\N$. 
Let $\sum_{j=1}^{\infty} a_j t^j$ be 
the Taylor series of $h$ at $t=0$, 
which converges on an open neighborhood of $t=0$.
Now, after the first finite sum of the Taylor series 
of $h$ and $F\circ\hat{\gamma}$ in (\ref{eqn:6.4})
are substituted into the equations
$(r\circ \gamma)(t)=
-2{\rm Re}(h(t))+
(F\circ\hat{\gamma})(t)=
O(t^N)$ for every $N\in\N$, 
we have
\begin{equation*}
{\rm Re}\left(
\sum_{j=1}^{N}a_j t^j 
\right)-
\sum_{j=2}^{N} {\rm Re}(c_j t^j)=O(t^{N+1}),
\end{equation*}
for every $N\in\N$. 
From the above equality, we can see that 
$c_j=a_j$ for every $j\in\N$. 
This means that $({\mathcal S}\circ\hat{\gamma})(t)$ 
converges 
and $({\mathcal S}\circ\hat{\gamma})(t)=h(t)$
on an open neighborhood of $t=0$.

\end{proof}


In \cite{FoN20}, 
smooth pseudoconvex real hypersurfaces in $\C^{n+1}$ 
of Bloom-Graham infinite type 
admitting no curve $\gamma_{\infty}$ in 
(\ref{eqn:1.2}) are constructed.
It follows from Theorem~6.3 that
many such hypersurfaces can be easily constructed. 

\begin{corollary}
Let $M$ be a real hypersurface defined by 
the equation (\ref{eqn:6.2}). 
If the Bloom-Graham type of $M$ at $0$ is infinity, 
then the following three conditions are equivalent. 
\begin{enumerate}
\item 
There is no $\gamma\in\Gamma$ 
tangent to $M$ at $0$ to infinite order; 
\item
For all $\hat{\gamma}\in\hat{\Gamma}$, 
the formal power series 
$({\mathcal S}\circ \hat{\gamma})(t)$ does
not converge at any point on a delated open neighborhood 
of $t=0$;
\item
The formal power series 
${\mathcal S}(z)$ does not converge
at any point on a delated open neighborhood 
of $z=0$.
\end{enumerate}
\end{corollary}

\begin{proof}
We remark that (ii) $\Longrightarrow$ (iii) can
be shown by using Abel's lemma 
(cf.  the proof of 
(iii) $\Longrightarrow$ (ii) in Theorem~6.3).   
The other implications 
can be directly obtained from Theorem~6.3.
\end{proof}

\begin{remark}
From (iii) in the above corollary, 
it is easy 
to construct smooth pseudoconvex hypersurfaces 
of Bloom-Graham infinite type 
satisfying the condition (i). 
One of simple examples of hypersurfaces is given 
by the equation
${\rm Re}(w)+
f_1(z_1,\bar{z}_1)+\cdots+f_n(z_n,\bar{z}_n)=0$,
where $f_j$ ($j=1,\ldots,n$) 
are subharmonic functions 
constructed in \cite{FoN20} 
(see also Remark~5.8).
The example constructed in \cite{FoN20}
takes the same form, but this example 
needs some additional conditions for each $f_j$. 

\end{remark}

Roughly speaking, 
when the flatness of hypersurfaces is stronger, 
it becomes easier to find the curve tangent to $M$ 
to higher order.
Thus, the following question seems to be
more difficult:
does there exist a smooth pseudoconvex real hypersurface 
in $\C^{n+1}$ ($n\geq 2$) 
of the Bloom-Graham {\it finite type} that 
admits no $\gamma\in\Gamma$ 
tangent to $M$ to infinite order?
The following theorem gives an affirmative answer.

\begin{theorem}
Let $n\geq 2$. 
There exists a smooth pseudoconvex real hypersurface 
$M$ in $\C^{n+1}$ with $\Delta_1(M,p)=\infty$ and
$\Delta_1^{\rm reg}(M,p)<\infty$
(in particular, $M$ is of Bloom-Graham finite type at $p$)
that 
admits no $\gamma\in\Gamma$
tangent to $M$ at $p$ to infinite order.
\end{theorem}

\begin{proof}
Notice that since a desired real hypersurface $M$ 
satisfies that 
$\Delta_1^{\rm reg}(M,p)<\Delta_1(M,p)$, 
there is no ${\mathcal N}$-canonical coordinate for $M$ 
from Theorem~2.3.

First, let us construct a desired real hypersurface $M$ in $\C^3$.
Let $\{c_j\}_{j\in\N}$ 
be an increasing sequence of positive real numbers
such that the power series $\sum_{j=8}^{\infty} c_j z^j$ 
does not converge away from the origin. 
Let $f$ be a real-valued smooth subharmonic function 
defined near the origin 
in $\C$ whose Taylor series is 
$2\sum_{j=8}^{\infty} {\rm Re}(c_j z^j)$
(see Remark~5.8). 
Let us consider the smooth function $F$ 
defined near the origin in $\C^2$ as 
\begin{equation}\label{eqn:6.5}
F(z,\bar{z})=\left|z_1^3-z_2^2\right|^2
+f(z_1,\bar{z}_1).
\end{equation}
Let $M$ be a smooth real hypersurface in $\C^{3}$ 
defined by 
$r(z,w)=2{\rm Re}(w)+F(z,\bar{z})=0$. 
The pseudoconvexity of $M$ is obvious. 

Now, let us consider $\Delta_1(M,0)$ and 
$\Delta_1^{\rm reg}(M,0)$ by using 
the Newton polyhedron of ${\mathcal N}_+(r)$.
It is easy to see 
$$
{\mathcal N}_+(r):=\{\xi\in\R_+^3:
\xi_1/6+\xi_2/4+\xi_3\geq 1\}.
$$
Let $\kappa_0:=
\{(\xi_1,\xi_2,0)\in{\mathcal N}_+(r):
\xi_1/6+\xi_2/4=1\}$, 
which is a compact face of ${\mathcal N}_+(r)$.
Then $F_{\kappa_0}$ is not ${\mathcal N}$-nondegenerate,  
while $F_{\kappa}$'s are  ${\mathcal N}$-nondegenerate
for the other compact faces $\kappa$. 
Therefore, 
Theorems~7.3 and 8.4 in \cite{Kam20}
implies that if ${\gamma}\in{\Gamma}$ does not
take the form 
\begin{equation}\label{eqn:6.6}
{\gamma}(t)=
(t^2,\pm t^3+o(t^3),o(t^{12})),
\end{equation}
then we have 
$$
1\leq 
\frac{{\rm ord}(r\circ{\gamma})}
{{\rm ord}({\gamma})}
\leq 6.
$$
In particular, every ${\gamma}\in{\Gamma}^{\rm reg}$
does not take the form (\ref{eqn:6.6}), so 
it is easy to see $\Delta_1^{\rm reg}(M,0)=6$.
More precisely, we consider the case when 
$\gamma$ takes the form (\ref{eqn:6.6}). 
\begin{lemma}
Let $N$ be an arbitrary integer with $N\geq 10$.
Then the following two conditions are equivalent.
\begin{enumerate}
\item ${\rm ord}(r\circ\gamma)\geq 2N+2$;
\item $\gamma(t)=(t^2,\pm t^3+O(t^{N+1}),
-\sum_{j=8}^N c_j t^{2j} +O(t^{2N+2})).$
\end{enumerate}
\end{lemma}
\begin{proof}
By dividing the function $r$ into a polynomial part 
and the remainder part, 
a defining function for $M$ can be rewritten 
as follows.
\begin{equation}\label{eqn:6.7}
r(z,w)=
2{\rm Re}\left(
w+\sum_{j=8}^N c_j z_1^j
\right)+
\left|z_1^3-z_2^2\right|^2
+R_{N+1}(z_1,\bar{z}_1),
\end{equation}
where 
$R_{N+1}\in C_0^{\infty}(\C)$ with 
${\rm ord}(R_{N+1})\geq N+1$. 

First, we show (i) $\Longrightarrow$ (ii).
It suffices to treat the curves of the form (\ref{eqn:6.6}),
which will be more exactly denoted by  
$\gamma(t)
=(t^2,\pm t^3+g(t),h(t))$ where 
$g\in{\mathcal O}_0(\C)$ with 
${\rm ord}(g)>3$ and 
$h\in{\mathcal O}_0(\C)$ with 
${\rm ord}(h)>12$.
Substituting this $\gamma(t)$ into (\ref{eqn:6.7}),  
we have
\begin{equation}\label{eqn:6.8}
(r\circ\gamma)(t)=
2{\rm Re}\left(h(t)+\sum_{j=8}^N c_j t^{2j}\right)+
|g(t)|^2 +R_{N+1}(t^2,\overline{t^2}).
\end{equation}
Since $R_{N+1}(t^2,\overline{t^2})=O(t^{2N+2})$
and  
the mixed terms and the pure terms cannot be canceled,
$g$ and $h$ must satisfy that  
$g(t)=O(t^{N+1})$ and
$h(t)+\sum_{j=8}^N c_j t^{2j}=O(t^{2N+2})$,
which imply the condition (ii).

Next, we show (ii) $\Longrightarrow$ (i). 
Substituting the equation in (ii) into (\ref{eqn:6.7}),
we can easily see 
${\rm ord}(r\circ\gamma)\geq 2N+2$.
\end{proof}

It follows from the above lemma that
$\Delta_1(M,p)=\infty$.

Now, let us assume that there exists a curve
$\gamma_{\infty}\in\Gamma$ such that 
${\rm ord}(r\circ\gamma_{\infty})\geq N$ 
for every $N\in\N$.
Let $\gamma_{\infty}(t)=:
(\gamma_{1}(t),\gamma_{2}(t),\gamma_{3}(t))$.
Since $\gamma_{\infty}$ satisfies the condition (i)
in Lemma~6.7, 
the condition (ii) implies that 
$\gamma_3(t)=-\sum_{j=8}^N c_j t^{2j} +O(t^{2N+2})$
for every $N\in\N$.
But, $\sum_{j=8}^{\infty} c_j t^{2j}$ does not converge
away from the origin, which is a contradiction to the holomorphy 
of $\gamma_3$.
As a result, we see that 
there exists no $\gamma\in\Gamma$ tangent to 
$M$ at the origin to infinite order.

In higher dimensional case $\C^{n+1}$ 
with $n\geq 3$, 
the following $F$ is considered: 
\begin{equation*}
F(z,\bar{z})=\left|z_1^3-z_2^2\right|^2
+f(z_1,\bar{z}_1)+\sum_{j=3}^n |z_j|^2,
\end{equation*} 
where $f$ is the same as that in (\ref{eqn:6.5}).
It is easy to construct higher dimensional 
hypersurfaces satisfying the properties in 
the theorem by using  $F$ in (\ref{eqn:6.5})
in a similar fashion to the three-dimensional case.
\end{proof}

\section{Open problems}

Theorems 5.5 and 6.3 only treat real hypersurfaces 
of the model form
as in (\ref{eqn:3.2}). 
The following problem is naturally raised. 

\begin{problem}
Let $n\geq 1$ and let $M$ be
a general smooth (pseudoconvex) real hypersurface 
in $\C^{n+1}$ and $p\in M$. 
Give equivalence conditions for 
the existence of $\gamma\in\Gamma$ 
tangent to $M$ at $p$ to infinite order,
analogous to those in Theorems 5.5 and 6.3.
\end{problem}

The second problem is concerned with Theorem~6.6. 

\begin{problem}
Let $n\geq 2$ and let $M$ be
a smooth (pseudoconvex) real hypersurface 
in $\C^{n+1}$ with 
$\Delta_1^{\rm reg}(M,p)<\infty$ for $p\in M$. 
Give equivalence conditions for 
the existence of $\gamma\in\Gamma$ 
tangent to $M$ at $p$ to infinite order.
\end{problem}

Theorem~6.6 only provides a simple example 
to the nonexistence of the desired $\gamma\in\Gamma$.

\vspace{ 2em}

{\bf Acknowledgements.} 

The author greatly appreciates that 
 Ninh Van Thu kindly informed of
important examples of 
real hypersurfaces of infinite type in \cite{FoN20}.
The author also would like to express his sincere
gratitude to the referee for his/her careful reading
of the manuscript and giving the
author valuable comments. 
This work was supported by 
JSPS KAKENHI Grant Numbers JP20K03656, 
JP20H00116.


\end{document}